\documentclass[12pt]{amsart}
\usepackage{blindtext}

\usepackage{amsmath,amsthm,amssymb,amsfonts}
\usepackage{color}
\usepackage[colorlinks=true,linkcolor=red,urlcolor=black]{hyperref}
\newtheorem{teo}{Theorem}
\newtheorem{prop}[teo]{Proposition}

\newtheorem{lema}[teo]{Lemma}

\newtheorem{maintheorem}{Theorem}
\newtheorem{maincor}[maintheorem]{Corollary}

\numberwithin{equation}{section}

\newcommand{\Hei}{\mathcal{H}}

\newcommand{\R}{\mathbb{R}}

\newcommand{\x}{{\bf x}}
\renewcommand{\v}{{\bf v}}

\renewcommand{\a}{{\bf a}}

\title[Lyapunov exponents for families of linear cocycles]{Lyapunov exponents for families of rotated linear cocycles}
\author[P. Valenzuela-Henr\'iquez]{Pancho Valenzuela-Henr\'iquez}
\address{Pancho Valenzuela-Henr\'iquez, Instituto de Matem\'atica,
Pontificia Universidad Cat\'olica de Valpara\'{\i}so, Blanco Viel 596,
Cerro Bar\'on, Valpara\'{\i}so-Chile.}
\email{francisco.valenzuela@ucv.cl}
\thanks{P.V.H. was partially supported by Proyecto Fondecyt 3120193}

\author[C. H. V\'asquez]{Carlos H. V\'asquez}
\address{Carlos H. V\'asquez, Instituto de Matem\'atica,
Pontificia Universidad Cat\'olica de Valpara\'{\i}so, Blanco Viel 596,
Cerro Bar\'on, Valpara\'{\i}so-Chile.} \email{carlos.vasquez@ucv.cl}
\thanks{C.H.V.  was supported by the Center of Dynamical Systems and Related Fields c\'odigo ACT1103 PIA - Conicyt and Proyecto Fondecyt 1130547.}
\date{\today}

\subjclass{Primary: 37H15, 37D25, 37D30.}

\keywords{partial hyperbolicity, Lyapunov exponents, cocycles}

\begin{document}

\begin{abstract}
In this work, we are interested in the study of the upper Lyapunov exponent $\lambda\sp+(\theta)$ associated to the periodic family of cocycles defined by
$$A_\theta(x):=A(x)R_\theta,\qquad x\in X,$$
where $A\::\: X\to \mathbb{GL}\sp+(2,\mathbb{R})$ is a linear cocycle orientation--preser\-ving and $R_\theta$ is a rotation of angle $\theta\in\mathbb{R}$.
We show that if the cocycle $A$  has dominated splitting, then
there exists a non empty open set $\mathcal{U}$ of parameters $\theta$ such that  the cocycle $A_\theta$ has dominated splitting and the function $\mathcal{U}\ni\theta\mapsto\lambda\sp+(\theta)$ is real analytic and strictly concave. As a consequence,  we obtain that the set of parameters $\theta$ where  the cocycle $A_\theta$ has not dominated splitting is non empty.
\end{abstract}

\maketitle


\section{Introduction}\label{sec:intro}
{Consider a compact metric space $X$}, let $T\::\:X\to X$ be an homeomorphism and let $\mu$ be an ergodic invariant measure for $T$. Let $A:X\to \mathbb{GL}(2,\mathbb{R})$ be continuous and consider the \emph{linear cocycle generated by $A$ over $T$}, $A_T:X\times\mathbb R\sp2\to X\times\mathbb R\sp2$, defined by

$$A_T(x,v)=(Tx,A(x)v).$$

\medskip

For $n\in\mathbb{Z}$, the iterates of $A_T$ are given by $A_T\sp n(x,v)=(T^n(x),A^n(x))$ where  $A\sp 0(x)=\rm Id$, and for any integer  $n>0$,

$$A\sp n(x)=A(T\sp{n-1}x)\cdot \ldots\cdot A(Tx)\cdot A(x),$$

$$A\sp{-n}(x)=A(T\sp{-n}x)\sp{-1}\cdot A(T\sp{-(n-1)}x)\sp{-1}\cdot \ldots\cdot A(T\sp{-1}x)\sp{-1}.$$

The \emph{upper Lyapunov exponent} of the cocycle $A_T$ {at $x\in X$} is defined by

$$\lambda^+(A_T,x)=\lim_{n\to+\infty}\frac{1}{n}\log\|A^n(x)\|.$$

In the same way, we define the \emph{lower Lyapunov exponent} of the cocycle $A_T$ {at  $x\in X$}  by

$$\lambda^-(A_T,x)=\lim_{n\to+\infty}\frac{1}{n}\log\|A^{-n}(x)\|\sp{-1}.$$

The classical theory of linear cocycles was initiated by Furstesberg and Kesten \cite{F1963, FK}. They proved that the limits above exist for almost every point $x\in X$ and they do not {depend on} the point when $\log\sp+\Vert A\sp{\pm1}\Vert\in L^1(\mu)$.  {More information is provided in} Oseledet's Theorem \cite{O,M} which establishes the existence of {a measurable} invariant splitting $\mathbb{R}^2=E^-(x)\oplus E^+(x)$ such that for $\mu$-almost all $x\in X$ and every $v\in E^+(x)\setminus\{0\}$
$$
\lim_{n\to\pm\infty}\frac{1}{n}\log\|A\sp n(x)v\|=\lambda^+(A_T,x).
$$
\noindent Similarly, for every $u\in E^-(x)$, $u\ne0$, we have
$$
\lim_{n\to\pm\infty}\frac{1}{n}\log\|A\sp n(x)u\|=\lambda^-(A_T,x).
$$

\medskip




We assume that the linear cocycle $A_T$ has \emph{dominated splitting}, that is, there exists a continuous decomposition $\mathbb{R}^2=E(x)\oplus F(x)$ such that for every $x\in X$:

\begin{enumerate}
\item $\dim E(x)=\dim F(x)=1,$
\item $A(x)E(x)=E(Tx)$, $A(x)F(x)=F(Tx)$, and
\item there exists $l\ge1$ such that
$$\|A\sp l(x)|{E(x)}\|\cdot\|A\sp{-l}(T\sp lx)|{F(T\sp lx)}\|<\frac12.$$
\end{enumerate}

We recall that in our case dominated splitting is a continuous extension of Osedelec splitting.

%
%

We denote by $\mathbb{GL}^+(2,\mathbb{R})$ the set of non singular matrices $A$ satisfying $\det A>0$. We are interested in the periodic family of cocycles defined by
$$
A_\theta(x):=A(x)R_\theta,\qquad x\in X,
$$
where $A\::\: X\to \mathbb{GL}\sp+(2,\mathbb{R})$ is a linear cocycle orientation--preserving and $R_\theta$ is a rotation of angle $\theta\in\mathbb{R}$.
For every $\theta\in\mathbb{R}$, the system underlying on the base $X$ is the original one $(T,X,\mu)$. For this reason, we will simplify the notation omitting $T$ if it is not necessary to explicit it.

When the matrices $A$ belong to $\mathbb{SL}(2,\mathbb{R})$ the previous family acquires interesting properties. For instance, Herman proved that the average of the upper Lyapunov exponent of $A_\theta$ is bounded above {by a certain value} involved with the norm of the original linear cocycle \cite[Section 6.2]{H}. Later, Avila and Bochi showed that the previous relation is in fact an equality \cite{AB}. The relation above was used by  Knill to  prove that {there exists} a dense  set of bounded $\mathbb{SL}(2,\mathbb{R})$ cocycles that have non-zero Lyapunov exponents (see \cite[Proposition 2.4, Theorem 3.1]{Kn92}). We refer the reader to \cite{BLD2011} for an alternative proof of the Herman-Avila-Bochi equality and \cite[Section 12.5.2]{BDV05}, \cite[Lemma 2]{Yoccoz2004} and \cite[Lemma 6.6]{BC2004} for another remarkable properties related with the family $A_\theta$.

\medskip

We are interested how the Lyapunov exponent $\lambda\sp+(A_\theta)$ varies with respect to $\theta$, when $A=A_0$ has dominated splitting. Earlier Ruelle \cite{Ruelle79} proved that the upper Lyapunov exponents varies analytically with respect to the cocycles with dominated splitting. On the other hand, concavity of the Lyapunov exponents was studied by Shub and Wilkinson \cite{SW} in the setting of perturbations of Anosov times identity on the three-torus. Roughly speaking, they studied how Lyapunov exponents change when the derivative $Df$ restricted to the center unstable subbundle $E^{cu}$ is composed by a matrix of type
$$
\left(\begin{array}{cc}
1&t\\0&1
\end{array}\right),\quad t\ne 0.
$$

\medskip

Our main Theorem is the following.

\begin{maintheorem}\label{teo:main}
Assume that the cocycle $A:X\to \mathbb{GL}^+(2,\mathbb{R})$ is continuous and has dominated splitting. Then, there exists an open set $\mathcal{U}\subseteq \mathbb{R}$ such that for any $\theta\in \mathcal{U}$ the cocycle $A_\theta$ has dominated splitting and the function $\mathcal{U}\ni\theta\mapsto\lambda\sp+(A_\theta)\in\R$ is real analytic and strictly concave.
\end{maintheorem}

In Theorem~\ref{teo:main} the set $\mathcal{U}$ is formed by the parameters $\theta \in \mathbb{R}$ such that $A_\theta$ has dominated splitting. However, it is not possible to have $\mathcal{U}=\mathbb{R}$ due the concavity. In fact, if we denote
$$
\mathcal D=\{\theta\in[0,2\pi]\::\: \mbox{ the cocycle }A_\theta \mbox{ has not dominated splitting}\}
$$
and we assume that $\mathcal D=\emptyset$, it follows from Theorem~\ref{teo:main} that the function $\theta\mapsto\lambda\sp+(AR_\theta)$ is real analytic, concave and periodic in the interval $[0,2\pi]$. In particular, the function $\theta\mapsto\lambda\sp+(A_\theta)$ must have a minimum in a point $\theta_0\in[0,2\pi]$. If the cocycle $B=AR_{\theta_0}$ has dominated splitting, by a change in the parameter, we can apply Theorem~\ref{teo:main} to the cocycle $B$ having a contradiction with the concavity. Summarizing,

\begin{maincor}\label{cor:mainnodomin}
The set $\mathbb{R}\setminus\mathcal{U}$ is not empty and for any $\theta\in\mathbb{R}\setminus\mathcal{U}$, the cocycle $A_\theta$ does not have dominated splitting.
\end{maincor}

Also as consequence of concavity, the Lyapunov exponents do not remain constant with respect to the parameter in the domain of domination. More precisely:

\begin{maincor}\label{cor:mainhyper}
There exists $\varepsilon>0$, an open interval $I\subseteq\mathbb{R}$, with $|I|=\varepsilon$ and $0\in\partial I$ such that  for any $\theta\in I$, we
have
$$\lambda\sp+(A_\theta)< \lambda\sp+(A)$$
and
$$\lambda\sp-(A_\theta)> \lambda\sp-(A).$$
\end{maincor}

The remaining of the paper is organized as follows. We will divide the proof of Theorem~\ref{teo:main} in two parts.

In Section~\ref{sec:QuasConj}, we introduce the notion of quasi--conjugation of cocycles. We prove that for quasi--conjugated cocycles,  their upper Lyapunov exponents are equals. We establish the existence of a linear cocycle $H$ quasi--conjugated to $A$.  The new cocycle $H$ is formed by upper triangular matrices  and it exhibits dominated splitting when $A$ has dominated splitting.

The core of the proof of Theorem~\ref{teo:main} is in Section~\ref{sec:pteo02}. There, we define a new metric that allows an easy calculation of the Lyapunov exponent. Then, inspired in the work of Shub and Wilkinson \cite{SW}, we give an explicit expression for the Lyapunov exponent $\lambda\sp+(A_\theta)$ and we show that the expressions involved are real analytic functions and thus, we can study the concavity of $\lambda\sp+(A_\theta)$.

\medskip

In Section~\ref{sec:appl}, we detailed an application of Theorem~\ref{teo:main}. Let $X=N$ be the compact nilmanifold obtained from the quotient of the Heisenberg group $\mathcal H$ with the lattice $\Gamma=\{(\x, y) : \x\in \mathbb{Z}^{2}, y\in \frac1{2}\mathbb{Z}\}$. Let $\Phi\::\:\mathcal H\to\Hei$ be an automorphism and $f:N\to N$ the diffeomorphism induced by $\Phi$. Then $f$ is partially hyperbolic with splitting
$$TN=E^s\oplus E^c\oplus E\sp u,$$
and the Lebesgue measure in $N$ is $f$-invariant and ergodic.

We consider the natural linear cocycle induced from $f$ and its  derivative $F:TN\to TN$ defined by
$$F((\x,y),v)=(f(\x,y),Df(\x,y)v).$$
In such a case, the cocycle defined by the restriction to the center-unstable direction $(F|E^{cu}):E\sp {cu}\to E\sp{cu}$ is orientation--preserving and it has dominated splitting. Moreover, the Lyapunov exponents are

$$-\lambda^u(F)=\lambda^s(F)<\lambda^c(F)=0<\lambda^u(F).$$
In particular, for the cocycle restricted to the center unstable direction we have
$$\lambda^+_c(F|E\sp{cu})=\lambda(Df|E^u)=\lambda\sp u(F)>0,$$
and
$$\lambda^-_c(F|E\sp{cu})=\lambda(Df|E^c)=\lambda\sp c(F)=0.$$

\noindent Then, we consider the one--parameter family of continuous cocycles $F_\theta\::\: TN\to TN$ defined by
$$
F_\theta((\x,y),v)=\begin{cases} Df(\x,y)R_\theta v&, \quad v\in E^{cu};\\
Df(\x,y)v &, \quad v\in
E^s.
\end{cases}
$$

\noindent Then we can apply Corollary~\ref{cor:mainhyper} to $F_\theta|E\sp{cu}$ and conclude the following.
\begin{maincor}\label{cor:main}
There is an open set $\mathcal{I}\subseteq [0,2\pi]$ such that for every $\theta\in \mathcal{I}$, the cocycle $F_\theta$ is partially hyperbolic with splitting $TN=E_\theta\sp s\oplus E_\theta\sp c \oplus E_\theta\sp u$ and $F_\theta$ is non uniformly hyperbolic.
\end{maincor}

\section{Quasi--conjugation}\label{sec:QuasConj}

Let $X$ be a compact metric space, let $T\::\:X\to X$ be an homeomorphism and let $\mu$ be an ergodic invariant measure for $T$. Consider the continuous maps $A:X\to \mathbb{GL}(2,\mathbb{R})$ and $H:X\to \mathbb{GL}(2,\mathbb{R})$ and consider the linear cocycle generated by $A$ over $T$, $A_T:X\times\mathbb R\sp2\to X\times\mathbb R\sp2$, and the linear cocycle generated by $H$ over $T$, $H_T:X\times\mathbb R\sp2\to X\times\mathbb R\sp2$.

The linear cocycles $A_T$, $H_T$ are {\em quasi--conjugated} if there exist two families of matrices $\mathfrak B=\{B(x)\in\mathbb{SO}(2,\mathbb R)\, :\, x\in X\}$ and $\mathfrak D=\{D(x)\in\mathbb{SO}(2,\mathbb R)\, :\, x\in X\}$ satisfying:
\begin{enumerate}
\item $H(x)=(D(x))\sp{-1}\cdot A(x)\cdot B(x)$, for every $x\in X$; and
\item $B(Tx)=s(x) D(x)$ where $s(x)\in\{1, -1\}$, for every $x\in X$.
\end{enumerate}

We remark that it is not required some continuity on the variable $x$ for the families of matrices. Recall that the element of $\mathbb{SO}(2,\mathbb R)$ are isometries of $\mathbb R\sp2$ and they form a commutative group.

We have the following.
\begin{lema}\label{le:qc001}
Let $A_T$ and $H_T$ be quasi--conjugated linear cocycles. Then
\begin{enumerate}
\item The cocycles $A_\theta$ and $H_\theta$ are quasi--conjugated,
\item $\lambda\sp+(A)=\lambda\sp+(H)$.
\end{enumerate}
\end{lema}
\begin{proof} Since $R_\theta\in\mathbb{SO}(2,\mathbb R)$ for every $\theta\in[0,2\pi]$ and $H(x)=(D(x))\sp{-1}\cdot A(x)\cdot B(x)$ for every $x\in X$, then we have
\[
\begin{aligned}
A_\theta(x)&=A(x)\cdot R_\theta\\
&=A(x)\cdot B(x)R_\theta(B(x))\sp{-1}\\
&=D(x)\left[(D(x))\sp{-1}\cdot A(x)\cdot B(x)\right]R_\theta(B(x))\sp{-1}\\
&=D(x)\cdot H(x)\cdot R_\theta\cdot(B(x))\sp{-1}\\
&=D(x)\cdot H_\theta(x)\cdot(B(x))\sp{-1}.
\end{aligned}
\]

On the other hand,  for each $n\in \mathbb N$ and $x\in X$ we have
\begin{equation}\label{eq:qc1}
\|A\sp n(x)\|=\|H\sp n(x)\|.
\end{equation}

In fact, arguing by induction it is not difficult to see that
\[
H\sp n(x)=\prod_{i=0}\sp{n-1}s(T\sp ix)\cdot (D(T\sp{n-1}x))\sp{-1}\cdot A\sp{n-1}(x)\cdot B(x)
\]
and therefore
$$
\|H\sp n(x)\|=\|(D(T\sp{n-1}x))\sp{-1}\cdot A\sp{n-1}(x)\cdot B(x)\|=\|A\sp{n}(x)\|.
$$

To conclude, it follows from Furstenberg--Kesten Theorem and \eqref{eq:qc1} that
\[
\begin{aligned}
\lambda\sp+(A)&=\lim_{n\to\infty}\frac{1}{n}\int\log\|A\sp n(y)\|d\mu(y)\\
&=\lim_{n\to\infty}\frac{1}{n}\int\log\|H\sp n(y)\|d\mu(y)\\
&=\lambda\sp+(H).
\end{aligned}
\]

\end{proof}

As an immediate consequence of Lemma~\ref{le:qc001}, if $A$ and $H$ are quasi--conjugated, then for every $\theta\in\mathbb{R}$ the cocycles $A_\theta$ and $H_\theta$ are quasi--conjugated.  Applying item (ii) of Lemma~\ref{le:qc001} to the quasi--conjugated cocycles $A_\theta$ and $H_\theta$ we obtain that
\begin{equation}\label{eq:quasiconj}
\lambda\sp+(A_\theta)=\lambda\sp+(H_\theta), \mbox{ for all }\theta \in\mathbb{R}.
\end{equation}

\begin{prop}\label{prop:qc1} If $A:X\to \mathbb{GL}(2,\R)$ is continuous and it has dominated splitting, then there exists $H:X\to \mathbb{GL}(2,\R)$ continuous which has dominated splitting  and for each $x\in X$, $H(x)$ is a lower triangular matrix. Moreover, $H_T$ is quasi--conjugated to $A_T$.
\end{prop}

First we will explicitly define the cocycle $H:X\to \mathbb{GL}(2,\R)$. Given a unitary vector $u=(u_1,u_2)\in\R\sp2$, we denote by $u\sp\bot$ the orthogonal vector $(u_2,-u_1)$. The column matrix
\[
B_u:=[u\sp\bot\, \, \, u]=\left(
                       \begin{array}{cc}
                         u_2 & u_1 \\
                         -u_1 & u_2 \\
                       \end{array}
                     \right)
\]
is a rotation and therefore  it belongs to $\mathbb{SO}(2,\mathbb R)$. Given $A\in \mathbb{GL}(2,\mathbb R)$ with
\[
A=\left(\begin{array}{cc}
                         a & b \\
                         c & d \\
                       \end{array}
                     \right),
\]
denote $v=Au/\|Au\|$ and $D_u=[v\sp\bot\, \, \, v]$. It is easy to see that
\[
H_u=D_u\sp{-1}\cdot A\cdot B_u=\left(
                                 \begin{array}{cc}
                                   \lambda_u & 0 \\
                                   \sigma_u & \eta_u \\
                                 \end{array}
                               \right)
\]
and $\eta_u>0$.

On the other hand, $A:X\to \mathbb{GL}(2,\mathbb{R})$ has dominated sppliting, so for every $x\in X$ let $\mathbb R\sp 2=E_A(x)\oplus F_A(x)$ be the dominated splitting of $A$. Let $\{V_1,\ldots, V_n\}$ be an open cover of $X$ such that there exist continuous functions $u_i:V_i\to \mathbb{R}\sp2$ such that $u_i(x)\in F_A(x)$ and $\|u_i(x)\|=1$ for every $x\in V_i$,  $i=1,\ldots, n$. For every $x\in V_i$, $i=1,\ldots, n$, we write
\begin{eqnarray*}
v_i(x)&:=&A(x)u_i(x)/\|A(x)u_i(x)\|,\\
B_i(x)&:=&[u_i\sp\bot(x)\, \, \, u_i(x)], \qquad\qquad \mbox{ and}\\
D_i(x)&:=&[v_i\sp\bot(x)\, \, \, v_i(x)].
\end{eqnarray*}

Define
\[
H_i(x)=(D_i(x))\sp{-1}\cdot A(x)\cdot B_i(x)=\left(
                                 \begin{array}{cc}
                                   \lambda_i(x) & 0 \\
                                   \sigma_i(x) & \eta_i(x) \\
                                 \end{array}
                               \right).
\]

\bigskip

\noindent Finally, define the function $H:X\to\mathbb{GL}(2,\mathbb R)$ by $H(x)=H_i(x)$ when $x\in V_i$.

\begin{lema}\label{le:qc002}
The function $H:X\to\mathbb{GL}(2,\mathbb R)$ is well defined and therefore it is continuous.
\end{lema}

\proof
Given $x\in V_i\cap V_j$ we have that
$$s_{ij}(x):=\langle u_i(x), u_j(x)\rangle=\pm1$$
and this is constant in each connected component of $V_i\cap V_j$. Therefore, it is easy to see that $B_j(x)=s_{ij}(x)B_i(x)$, $D_j(x)=s_{ij}(x)D_i(x)$ and $D_j(x)\sp{-1}=s_{ij}(x)D_i(x)\sp{-1}$ for every $x\in V_i\cap V_j$. Then we have
\[
\begin{aligned}
H_j(x)&=(D_j(x))\sp{-1}\cdot A(x)\cdot B_j(x)\\
&=s_{ij}(x)(D_i(x))\sp{-1}\cdot A(x)\cdot s_{ij}(x)B_i(x)\\
&=(s_{ij}(x))\sp2(D_i(x))\sp{-1}\cdot A(x)\cdot B_i(x)\\
&=H_i(x).
\end{aligned}
\]
\endproof

We can write the function $H:X\to\mathbb{GL}(2,\mathbb R)$ by
\[
H(x)=\left(
       \begin{array}{cc}
         \lambda(x) & 0 \\
         \sigma(x) & \eta(x) \\
       \end{array}
     \right)
\]
with $\eta(x)>0$. We also denote
\[
H\sp n(x)=\left(
            \begin{array}{cc}
              \lambda_n(x) & 0 \\
              \sigma_n(x) & \eta_n(x) \\
            \end{array}
          \right)
\]
where
\begin{equation}\label{eqp:11}
\lambda_n(x)=\prod_{j=0}\sp{n-1}\lambda(T\sp jx)\, \, \, \textrm{and}\, \, \, \eta_n(x)=\prod_{j=0}\sp{n-1}\eta(T\sp jx).
\end{equation}

\bigskip

\noindent
\begin{lema}\label{le:cq003} There is an continuous splitting $\mathbb R\sp2=E(x)\oplus F(x)$ invariant by  $H_T$.
\end{lema}
\begin{proof}Define $F(x)=\rm{span}(\mathbf e_2)$ and $E(x)=(B(x))\sp{-1}E_A(x)$ where $\mathbf e_2=(0,1)$. Clearly $H(x)\left(F(x)\right)=F(Tx)$.
Note that

\begin{eqnarray*}
E(Tx)=(B(x))\sp{-1}E_A(Tx)&=&{\rm span}\left((B(Tx))\sp{-1}E_A(Tx)\right)\\
&=&{\rm span}\left(s(x)(D(x))\sp{-1}E_A(Tx)\right)\\
&=&(D(x))\sp{-1}E_A(Tx).
\end{eqnarray*}

\noindent If $v\in E_A(x)$, then $(B(x))\sp{-1}v\in E(x)$. The equality
\[
w=A(x)v=D(x)\cdot H(x)\cdot (B(x))\sp{-1}v\in E_A(Tx)
\]
implies that $(D(x))\sp{-1}w=H(x)(B(x))\sp{-1}v\in E(Tx)$ as required.
\end{proof}

\begin{lema}\label{le:qc004} For each $x\in X$, $\eta(x)>|\lambda(x)|$.
\end{lema}

\begin{proof} Fix $x\in X$ and let $u\in F(x)$ and $v\in E(x)$ be unitary vectors. Since
\begin{equation}
\det A(x)=\|A(x)v\|\cdot \|A(x)u\|\cdot \sin(\angle(A(x)v,A(x)u))
\end{equation}
we have that
\[
\begin{aligned}
\frac{\|A(x)|{E(x)}\|}{\|A(x)|{F(x)}\|}&=\frac{\|A(x)v\|\cdot\|A(x)u\|}{\|A(x)u\|\sp2}\\
&=\frac{|\det(A(x))|}{\|A(x)u\|\sp2\sin(\angle(A(x)v,A(x)u))}\\
&<1.
\end{aligned}
\]
Hence, we conclude that
\[
|\det(A(x))|<\|A(x)u\|\sp2\sin(\angle(A(x)v,A(x)u))<\|A(x)u\|\sp2.
\]
Let $u$ be such that $(B(x))\sp{-1}u=\mathbf e_2$. Since $D(x)$ is an isometry and $B(Tx)=s(x) D(x)$ where $s(x)\in\{1, -1\}$, for every $x\in X$, we conclude that
\[
\|A(x)u\|=\|D(x)\cdot H(x)\cdot (B(x))\sp{-1}u\|=\|H(x)\mathbf e_2\|=\eta(x),
\]
and therefore
\[
|\det(A(x))|=|\det(H(x))|=|\lambda(x)|\cdot\eta(x)<(\eta(x))\sp2
\]
as desired.
\end{proof}

\begin{lema}\label{le:cq005} The splitting $\mathbb R\sp2=E(x)\oplus F(x)$ is dominated.
\end{lema}
\begin{proof} We remark that by construction, there exists $\alpha>0$ such that for every $x\in X$,
$$\angle(E(x), F(x))>\alpha.$$
Therefore, there exists $C>0$ such that
$$\frac1{\sin\left(\angle(E(x), F(x))\right)}<C.$$

\noindent Taking $v\in E(x)$ unitary and $\mathbf e_2\in F(x)$, it follows that
\[
\begin{aligned}
\frac{\|H\sp n(x)|{E(x)}\|}{\|H\sp n(x)|{F(x)}\|}&=\frac{\|H\sp n(x)v\|}{\|H\sp n(x)\mathbf e_2\|}\\
&=\frac{\|H\sp n(x) v\|\cdot\|H\sp n(x)\mathbf e_2\|}{\|H\sp n(x)\mathbf e_2\|\sp2}\\
&=\frac{|\det(H\sp n(x))|}{\|H\sp n(x)\mathbf e_2\|\sp2\sin(\angle(H\sp n(x) v,H\sp n(x)\mathbf e_2))}\\
&<C\cdot\frac{|\det(H\sp n(x))|}{\|H\sp n(x)\mathbf e_2\|\sp2}\\
&=C\cdot\frac{|\lambda_n(x)|}{\eta_n(x)}.
\end{aligned}
\]
From Lemma~\ref{le:qc004},  $|\lambda_n(x)|/\eta_n(x)\to 0$ when $n\to\infty$. This implies that
\[
\frac{\|H\sp n(x)|{E(x)}\|}{\|H\sp n(x)|{F(x)}\|}<C\cdot\frac{|\lambda_n(x)|}{\eta_n(x)}<1
\]
for $n$ large enough, and therefore $H$ has dominated splitting as desired.
\end{proof}

\begin{lema}\label{le:qc006}  The cocycle $H_T$ is quasi--conjugated to $A_T$.
\end{lema}

\begin{proof}
Let $x\in V_i$ with $Tx\in V_j$. Note that by construction there exists $s_{ij}(x)\in\{1,-1\}$ such that $v_i(x)=s_{ij}(x)\cdot u_j(Tx)$ and therefore $D_i(x)=s_{ij}(x)B_j(Tx)$. We define the families $\mathfrak B$ and $\mathfrak D$ inductively as follows:
\begin{enumerate}
\item for each $x\in V_1$, we let $B(x):=B_1(x)$ and $D(x):=D_1(x)$,
\item for each $x\in V_i\setminus \cup_{j<i}V_j$, we let $B(x):=B_i(x)$ and $D(x):=D_i(x)$
\end{enumerate}
and our claim it follows from the construction above.
\end{proof}


\section{Proof of Theorem~\ref{teo:main}.}\label{sec:pteo02}

This section is devoted to prove  Theorem~\ref{teo:main}. We remark that Theorem~\ref{teo:main} follows from the next result.

\begin{teo}\label{teo:part2}
 Assume that the cocycle $A:X\to \mathbb{GL}^+(2,\mathbb{R})$ is continuous and has dominated splitting. Then, there exists $\varepsilon>0$ such that:
\begin{enumerate}
\item There exist one--parametric continuous functions
$$a_\theta, c_\theta, u_\theta:X\to\mathbb R,\qquad\theta\in(-\varepsilon,\varepsilon),$$
such that for all $\theta\in(-\varepsilon, \varepsilon)$, we have
\begin{eqnarray*}
\lambda\sp+(A_\theta)&=&\lambda\sp+(A)+\lambda\sp-(A)\\
& &-\int_X\log\left(a_\theta(x)-c_\theta(x)u_\theta(Tx)\right)d\mu(x).
\end{eqnarray*}

\item For each $x\in X$, the functions $(-\varepsilon,\varepsilon)\ni\theta\mapsto a_\theta(x), c_\theta(x), u_\theta(x)$ are real analytic.
\item The function $\lambda\sp+:(-\varepsilon,\varepsilon)\to\mathbb R$ given by $\lambda\sp+(\theta):=\lambda\sp+(A_\theta)$ is real analytic and
\[
\frac{d\sp2\lambda\sp+(0)}{d\theta\sp2}<0.
\]
\end{enumerate}
\end{teo}

From Proposition~\ref{prop:qc1}, it is enough to give the proof for the cocycle $H_\theta$.

\subsection{Implicit expression for the Lyapunov exponent.}\label{ssp:imp_exp}

Fix $\varepsilon>0$ such that $ H_\theta$ has dominated splitting. We denote the invariant splitting of $H_\theta$ by $E_\theta(x)\oplus F_\theta(x)=\mathbb R\sp2$. Since $F_0(x)={\rm span}(\mathbf e_2)$, there exist a continuous function $(-\varepsilon,\varepsilon)\times X\ni(\theta,x)\mapsto u_\theta(x)$ such that $F_\theta(x)={\rm span}\left((u_\theta(x),1)\right)$. Therefore by construction $u_0(x)=0$ for each $x\in X$.

\medskip

Consider $X\times\mathbb R\sp2=\cup_{x\in X}\{x\}\times\mathbb R\sp2$ and denote $T_x=\{x\}\times\mathbb R\sp2$. In all the preceding, we consider the standard inner product over $T_x\cong\mathbb R\sp2$. For $(x,\theta)$, we consider $\langle\cdot,\cdot\rangle_{x,\theta}$ the inner product over $T_x$ such that the set $\left\{(1,-u_\theta(x)), (u_\theta(x),1)\right\}$ is an orthonormal base. Then, for each $\theta\in[0,2\pi]$ and $u,v\in\mathbb R\sp2$ the map $x\mapsto\langle u,v\rangle_{x,\theta}$ is continuous. Since $X$ is compact and all metric in $\mathbb R\sp2$ are equivalent, we can choose a constant $C>0$ such that $C\sp{-1}\|\cdot\|_{x,\theta}\le \|\cdot\|\le C\|\cdot\|_{x,\theta}$ for all $x\in X$. Moreover, for every $\theta\in[0,2\pi]$, we have

\begin{equation}\label{eqp:04}
\lambda\sp+(\theta):=\lambda\sp+(H_\theta)=\int\log\|H_\theta(y)|{F_\theta(y)}\|_{Ty,\theta}d\mu(y).
\end{equation}

In fact, it follows from Birkhoff's Theorem that for $\mu$--a.e. $x\in X$ we have

\begin{eqnarray*}
\lambda\sp+(\theta)&=&\lim_{n\to\infty}\frac{1}{n}\log\|H\sp n(x)(u_\theta(x),1)\|\\
&=&\lim_{n\to\infty}\frac{1}{n}\log\|H\sp n(x)(u_\theta(x),1)\|_{T\sp nx,\theta}\\
&=&\lim_{n\to\infty}\frac{1}{n}\sum_{j=0}\sp{n-1}\log\|H_\theta(T\sp jx)|{F_\theta(T\sp jx)}\|_{T\sp{j+1}x,\theta}\\
&=&\int\log\|H_\theta(y)|{F_\theta(y)}\|_{Ty,\theta}d\mu(y).
\end{eqnarray*}

Item (i) of Theorem~\ref{teo:part2} follows directly from the next result.

\begin{prop}\label{prop:TD01}
For every $\theta\in[0,2\pi]$ define
\[
H_\theta(x)=\left(
                 \begin{array}{cc}
                   a_\theta(x) & b_\theta(x) \\
                   c_\theta(x) & d_\theta(x) \\
                 \end{array}
               \right)\quad\textrm{for all $x\in X$.}
\]
Then, there exists $\varepsilon>0$ such that for $|\theta|<\varepsilon$, we have
\begin{equation}\label{eqp:05}
\lambda\sp+(\theta)=\lambda\sp+(0)+\lambda\sp-(0)-\int_X\log(a_\theta(x)-
c_\theta(x) u_\theta(Tx))d\mu(x).
\end{equation}

\end{prop}

\begin{proof} For every $\theta\in[0,2\pi]$ and  every $x\in X$, define
$$r_\theta(x)=\|H_\theta(x)|{F_\theta(x)}\|_{Tx,\theta}.$$

Then, equation~(\ref{eqp:04}) implies that
\[
\lambda\sp+(\theta)=\int\log r_\theta(x)d\mu(x).
\]
Note that there exists $t_\theta(x)\in\mathbb R$ such that
\[
(HR_\theta)(x)(u_\theta(x),1)=t_\theta(x)(u_\theta(Tx),1)
\]
and hence
\[
\|(HR_\theta)(x)(u_\theta(x),1)\|_{Tx,\theta}=|t_\theta(x)|\cdot\|(u_\theta(Tx),1)\|_{Tx,\theta}
=|t_\theta(x)|=r_\theta(x).
\]
Writing

\begin{eqnarray}
(HR_\theta)(x)&=&\left(
                 \begin{array}{cc}
                   \lambda(x)\cos\theta & -\lambda(x)\sin\theta \\
                   \sigma(x)\cos\theta+\eta(x)\sin\theta & \eta(x)\cos\theta-\sigma(x)\sin\theta\\
                 \end{array}
               \right)\nonumber \\
&=&\left(
                 \begin{array}{cc}
                   a_\theta(x) & b_\theta(x) \\
                   c_\theta(x) & d_\theta(x) \\
                 \end{array}
               \right),\label{eq:relprinc}
\end{eqnarray}

we have the equations
\begin{equation}\label{eqp:06}
\begin{cases}
a_\theta(x)u_\theta(x)+b_\theta(x)&=\, \, t_\theta(x)u_\theta(Tx)\\
c_\theta(x)u_\theta(x)+d_\theta(x)&=\, \, t_\theta(x)
\end{cases}
\end{equation}
and multiplying the first equation by $-c_\theta(x)$ and the second one by $a_\theta(x)$ we
deduce that
\[
t_\theta(x)=\frac{\eta(x)\lambda(x)}{a_\theta(x)-c_\theta(x)u_\theta(Tx)}.
\]
Since $\lambda>0$ and $u_0=0$, it follows that $a_\theta(x)-c_\theta(x)u_\theta(Tx)>0$ for $\theta$ close to 0 and for each $x\in X$. Then, there exists $\varepsilon>0$ such that if $\theta\in(-\varepsilon, \varepsilon)$ then $t_\theta(x)>0$. We conclude

\begin{eqnarray*}
\lambda\sp+(\theta)&=&\int\log t_\theta(x)d\mu(x)\\
&=&\int\log\eta(x)d\mu(x)+\int\log\lambda(x)d\mu(x)\\
& & -\int_X\log \left(a_\theta(x)-c_\theta(x)u_\theta(Tx)\right)d\mu(x).
\end{eqnarray*}

Recalling that
\[
\lambda\sp+(0)=\int\log\|H(x)|_{F(x)}\|d\mu(x)=\int\log\eta(x) d\mu(x)
\]
{and}
\[
\lambda\sp+(0)+\lambda\sp-(0)=\lim_{n\to\infty}\frac{1}{n}\log|\det(H\sp n(x))|=\int\log(\lambda(x)\eta(x))d\mu(x)
\]
for $\mu$--a.e. $x\in X$, we obtain

$$\lambda\sp+(\theta)=\lambda\sp+(0)+\lambda\sp-(0)-\int_X\log
\left(a_\theta(x)-c_\theta(x)u_\theta(Tx)\right)d\mu(x).$$
\end{proof}


\subsection{Analyticity of the Lyapunov exponent.}\label{ssp:holm_fam} This subsection is devoted to explain the main tools to prove that the map $\theta\mapsto u_\theta(x)$ is real analytic.

For the real matrix
\[
A=\left(\begin{array}{cc}
                         a & b \\
                         c & d \\
                       \end{array}
                     \right),
\]
we denote the action of $A$ in the Riemann $\overline{\mathbb C}$ by a M\"{o}bius transformation by
\[
A\cdot z=\frac{az+b}{cz+d},\quad z\in \overline{\mathbb C}.
\]

\begin{lema}\label{lmp:01}
A linear cocycle $A$ has dominated splitting if and only if there exist a family of open disks $\mathcal D=\{\mathbb D(x)\}_{x\in X}$ with $\mathbb D(x)\subset \overline{\mathbb C}$ and $N\ge0$ such that \begin{equation}\label{eq:01}
\overline{A\sp n(x)\cdot\mathbb D(x)}\subset\mathbb D(T\sp nx)
\end{equation}
for each $x\in X$ and $n\ge N$.
\end{lema}
\proof The reader can find the complete proof of this Lemma  in \cite{pancho2}. Nevertheless, we explain the main steps by completeness.

For the sufficient direction, a classical argument of contraction for family of cones, gives us the existence of the family of disks.

For the reciprocal, first we observe that the linear cocycle $A$ over $X\times \mathbb C\sp2$ has dominated splitting if and only if the projective action $A(x)\cdot[v]=[A(x)v]$ is hyperbolic on $X\times\overline{\mathbb C}$. More precisely, there exist sections $\tau, \sigma:X\to \overline{\mathbb C}$ and constants $C>0$ and $0<\lambda<1$ such that
\begin{enumerate}
\item $A(x)\cdot\tau(x)=\tau(Tx)$ and $A(x)\cdot\sigma(x)=\sigma(Tx)$,
\item $\|(A\sp n(x))\sp\prime\cdot \tau(x)\|\le C\lambda\sp n$ for each $n\ge0$, and
\item $\|(A\sp{-n}(x))\sp\prime\cdot\sigma(x)\|\le C\lambda\sp n$ for each $n\ge0$,

\end{enumerate}
where the norm  is provided by the spherical metric. Moreover, we can show that in order to obtain dominated splitting it is enough to exhibit a contractive or an expansive section (the $\tau$ or $\sigma$ section respectively).

Actually, it is sufficient the existence of $\tau_x\in \overline{\mathbb C}$ such that

\begin{equation}\label{eq:hypdisc}
\|(A\sp n(x))\sp\prime\cdot\tau_x\|\le C\lambda\sp n.
\end{equation}

This follows from the fact that a contractive direction must be unique, and therefore the correspondence $x\mapsto\tau_x$ is continuous. Finally, we observe that the condition of contraction of disks and Schwartz Lemma give us the existence of  $\tau_x\in\mathbb D(x)$ for each $x\in X$ such that $\tau_x$ satisfies \eqref{eq:hypdisc}.

\endproof

Let $A$ be a linear cocycle with dominated splitting $E\oplus F=\mathbb R\sp2$. For $F(x)={\rm span}(v)$ with $v=(v_1,v_2)\in\mathbb R\sp2$ and $v_2\neq2$, we define $\xi(x)=v_1/v_2$ and for $F(x)={\rm span}((1,0))$ define $\xi(x)=\infty$. Therefore
\[
\xi(x)=\bigcap_{n\ge0}\overline{A\sp n(T\sp{-n}x)\cdot\mathbb  D(T\sp{-n}x)}.
\]
Moreover, the function $\xi:X\to\overline{\mathbb C}$ is continuous and $\xi(x)\in\mathbb  D(x)$.

Let $\Lambda\subset\mathbb C$ be an open connected set. We say that the family of cocycles $\{A_\lambda\}_{\lambda\in \Lambda}$ is {\em holomorphic} if  for every $x\in X$ and $\lambda \in\Lambda$
\[
A_\lambda(x)=\left(
  \begin{array}{cc}
    a_{11}(x,\lambda) & a_{12}(x,\lambda) \\
    a_{21}(x,\lambda) & a_{22}(x,\lambda) \\
  \end{array}
\right),
\]
and the functions $\lambda\mapsto a_{ij}(x,\lambda)$ for $i, j=1,2$ are holomorphic for every $x\in X$.

The family $\{A_\lambda\}_{\lambda\in \Lambda}$ has dominated splitting if for each $\lambda\in\Lambda$, $A_\lambda$ has dominated splitting. In this setting, we consider the correspondence $\lambda\mapsto\xi_\lambda(x)$ as before.
\begin{prop}\label{lmp:02}
If the family $\{A_\lambda\}_{\lambda\in \Lambda}$ has dominated splitting, then for each $x\in X$ the map $\lambda\mapsto\xi_\lambda(x)$  is holomorphic.
\end{prop}
\begin{proof}
The proof follows from Montel's theorem. We recall the basic fact.

A family $\mathcal F$ of holomorphic functions defined over a fixed domain $\Lambda\subset\mathbb C$ is said to be normal if every sequence of members of $\mathcal F$ has a subsequence that converges uniformly on compact subsets of $\Lambda$. We recall that if a sequence of holomorphic functions converge uniformly on compact sets, then the limit function is also holomorphic. Finally, Montel theorem assert that a family $\mathcal F=\{f:\Lambda\to \mathbb D\}$ of holomorphic functions is normal.

Fix $\lambda\in\Lambda$. We recall that if $\mathcal D$ is the family of disks given by Lemma~\ref{lmp:01} for the cocycle $A_\lambda$, then the equation~(\ref{eq:01}) hold for $A_\zeta$ for $\zeta\in D(\lambda)$ in an small disk around $\lambda$, with the same family of disks.

Fix $x\in X$. Note that the function $f_n(\zeta)=A_\zeta\sp n(T\sp{-n}x)\cdot\xi_\lambda(T\sp{-n}x)$ for $n\in\mathbb N$ is holomorphic. From equation~(\ref{eq:01}) we have that
\[
\overline{A\sp n(T\sp{-n}x)\cdot\mathbb  D(T\sp{-n}x)}\subset\mathbb  D(x)
\]
and since that $\xi_\lambda(T\sp{-n}x)\in\mathbb  D(T\sp{-n}x)$ we conclude that $f_n(D(\lambda))\subset\mathbb  D(x)$. It follows from Montel Theorem that $\{f_n:D(\lambda)\to \mathbb D(x)\, :\, n\in\mathbb N\}$ is a normal family. Finally, taking a subsequence if necessary, from equation~(\ref{eq:01}) we have that $f_n(\zeta)\to\xi_\zeta(x)$ uniformly on compact sets, as required.
\end{proof}

The next proposition summarize the proof of item (ii) and the first statement of item (ii) in Theorem~\ref{teo:part2}.

\begin{prop}\label{prop:anal} For each $x\in X$, the functions $(-\varepsilon,\varepsilon)\ni\theta\mapsto a_\theta(x)$, $ c_\theta(x)$, $u_\theta(x)$ and the function $\lambda\sp+:(-\varepsilon,\varepsilon)\to\mathbb R$ given by $\lambda\sp+(\theta):=\lambda\sp+(AR_\theta)$ are real analytic.
\end{prop}

\proof Since $a_\theta(x)=\lambda(x)\cos\theta$ and $c_\theta(x)=\sigma(x)\cos\theta+\eta(x)\sin\theta$, it only remains to prove this fact for the function $u_\theta(x)$.

Let $\varepsilon>0$ as in Proposition~\ref{prop:TD01} and let $U(1,r)=\{z\in\mathbb C\,:\, |1-z|<r<1\}$ such that $\mathbb S\sp1\cap U(1,r)\subset\{e\sp{i\theta}\,:\, |\theta|<\varepsilon\}$. For each $z\in U(1,r)$, define
\[
S_z=\left(
      \begin{array}{cc}
        \frac{z+z\sp{-1}}{2} & -\frac{z-z\sp{-1}}{2} \\
        \frac{z-z\sp{-1}}{2} & \frac{z+z\sp{-1}}{2} \\
      \end{array}
    \right).
\]
We have that $S_{e\sp{i\theta}}=R_\theta$. Define $ H_z=(H_z)_T$ where
$H_z(x)=H(x)S_z$. Then for $0<r\ll1$ the family $\mathcal F=\{H_z\, :\, z\in U(1,r)\}$ has dominated splitting and therefore, for each $x\in X$ the map $z\mapsto\xi_z(x)$ is holomorphic. Our assertion it follows from the fact $\xi_{e\sp{i\theta}}(x)=u_\theta(x)$. The assertion of analyticity corresponding to the function $\lambda\sp+(\theta):=\lambda\sp+(AR_\theta)$ is immediate from \eqref{eqp:05}.

\endproof
\subsection{Calculating the derivatives of $\lambda\sp+(\theta)$}\label{ssec:calculo}

From Proposition~\ref{prop:anal} we know that the function $\lambda\sp+:(-\varepsilon,\varepsilon)\to\mathbb R$ defined by $\lambda\sp+(\theta):=\lambda\sp+(AR_\theta)$ is real analytic. Now we proceed to calculate the derivatives of the Lyapunov exponent $\lambda\sp+(\theta)$ in $\theta=0$ and to study the concavity finishing the proof of item (iii) in Theorem~\ref{teo:part2}.

\begin{lema}\label{le:cl001}
For every $x\in X$ we have
\begin{equation}\label{eq:pderl}
\frac{d\lambda\sp+(0)}{d\theta}=
\int\frac{\sigma(x)}{\lambda(x)}\dot u_0(Tx)d\mu(x)
\end{equation}

and
\begin{eqnarray}\label{eq:sderl}
\frac{d\sp2\lambda\sp+(0)}{d\theta\sp2}&=&\int\left[\frac{2\eta(x)\dot u_0(Tx)}{\lambda(x)}+1+\frac{\sigma(x)\ddot u_0(Tx)}{\lambda(x)}\right.\\ \nonumber
& &+\left.\left(\frac{\sigma(x)\dot u_0(Tx)}{\lambda(x)}\right)\sp2\right]d\mu(x).
\end{eqnarray}
\end{lema}

\begin{proof}
From equation~(\ref{eqp:05}) it follows that $\lambda\sp+$ is differentiable and
\[
\begin{aligned}
\frac{d\lambda\sp+(\theta)}{d\theta}&=-\int\frac{d}{d\theta}
\log\left[a_\theta(x)-c_\theta(x)u_\theta(Tx)\right]
d\mu(x)\\
&=-\int\frac{\dot a_\theta(x)-\dot c_\theta(x) u_\theta(Tx)-c_\theta(x) \dot u_\theta(Tx)}{a_\theta(x)-c_\theta(x)u_\theta(Tx)}
d\mu(x)\\
&=\int\frac{\dot c_\theta(x) u_\theta(Tx)+c_\theta(x) \dot u_\theta(Tx)-\dot a_\theta(x)}{a_\theta(x)-c_\theta(x)u_\theta(Tx)}
d\mu(x).
\end{aligned}
\]
Since $u_0(x)=0$ for every $x\in X$, we conclude
\[
\frac{d\lambda\sp+(0)}{d\theta}=
\int\frac{c_0(x) \dot u_0(Tx)-\dot a_0(x)}{a_0(x)}
d\mu(x).
\]

 From \eqref{eq:relprinc} we have $a_\theta(x)=\lambda(x)\cos\theta$ and  $c_\theta(x)=\sigma(x)\cos\theta+\eta(x)\sin\theta$, for every $x\in X$. A simple calculation allow us to conclude that

\smallskip

\[
\frac{d\lambda\sp+(0)}{d\theta}=\int\frac{\sigma(x)}{\lambda(x)}\dot u_0(Tx)d\mu(x).
\]
Similarly, we can show that

\begin{eqnarray*}
\frac{d\sp2\lambda\sp+(\theta)}{d\theta\sp2}&=&
\int\left[\frac{\ddot c_\theta(x) u_\theta(Tx)+2\dot c_\theta(x)\dot u_\theta(Tx)+c_\theta(x)\ddot u_\theta(Tx)-\ddot a_\theta(x)}{a_\theta(x)-c_\theta(x)u_\theta(Tx)}\right.\\
& & +\left.\left(\frac{\dot c_\theta(x) u_\theta(Tx)+c_\theta(x) \dot u_\theta(Tx)-\dot a_\theta(x)}{a_\theta(x)-c_\theta(x)u_\theta(Tx)}\right)\sp2\right]d\mu(x)
\end{eqnarray*}

\noindent and therefore

\begin{eqnarray*}
\frac{d\sp2\lambda\sp+(0)}{d\theta\sp2}&=&
\int\left[\frac{2\dot c_0(x)\dot u_0(Tx)+c_0(x)\ddot u_0(Tx)-\ddot a_0(x)}{a_0(x)}\right.\\
& & \left.+\left(\frac{c_0(x) \dot u_0(Tx)-\dot a_0(x)}{a_0(x)}\right)\sp2\right]d\mu(x)\\
&=&\int\left[\frac{2\eta(x)\dot u_0(Tx)}{\lambda(x)}+1+\frac{\sigma(x)\ddot u_0(Tx)}{\lambda(x)}+\left(\frac{\sigma(x)\dot u_0(Tx)}{\lambda(x)}\right)\sp2\right]d\mu(x).
\end{eqnarray*}

\end{proof}

\begin{lema}\label{prp:002}
For each $x\in X$, we have
\begin{equation}\label{eqp:12}
\dot u_0(x)=-\sum_{k=1}\sp \infty\frac{\lambda_k(T\sp{-k}x)}{\eta_k(T\sp{-k}x)}
\end{equation}
and
\begin{equation}\label{eq:sdu}
\ddot u_0(x)=-\sum_{k=1}\sp \infty\frac{\lambda_k(T\sp{-k}x)}{\eta_k(T\sp{-k}x)}\cdot\alpha(T\sp{-k}x)\cdot\sigma(T\sp{-k}x)
\end{equation}
where
\[
\alpha(x)=\frac{2\left(\dot u_0(x)-1\right)\sp2}{\eta(x)}
\]
and $\lambda_k$, $\eta_k$ are as in equation~(\ref{eqp:11}).
\end{lema}

\begin{proof} The relations in (\ref{eqp:06}) establish that, for every $\theta\in[0,2\pi]$ and  every $x\in X$,
\begin{equation}\label{eq:der1}
u_\theta(x)=\frac{a_\theta(T\sp{-1}x)u_\theta(T\sp{-1}x)+b_\theta(T\sp{-1}x)}
{c_\theta(T\sp{-1}x)u_\theta(T\sp{-1}x)+d_\theta(T\sp{-1}x)}.
\end{equation}
Regarding that from \eqref{eq:relprinc} we have $a_\theta(x)=\lambda(x)\cos\theta$, $b_\theta(x)=-\lambda(x)\sin\theta$, $c_\theta(x)=\sigma(x)\cos\theta+\eta(x)\sin\theta$ and $d_\theta(x)=\eta(x)\cos\theta-\sigma(x)\sin\theta$, define
\begin{eqnarray}\label{eq:f}
f_\theta(x)&=&a_\theta(T\sp{-1}x)u_\theta(T\sp{-1}x)+b_\theta(T\sp{-1}x)\nonumber\\
&=& \lambda(T\sp{-1}x)\cos\theta\cdot u_\theta(T\sp{-1}x)-\lambda(T\sp{-1}x)\sin\theta
\end{eqnarray}
and

\begin{eqnarray}\label{eq:g}
g_\theta(x)&=&c_\theta(T\sp{-1}x)u_\theta(T\sp{-1}x)+d_\theta(T\sp{-1}x)\nonumber\\
&=&\left(\sigma(T\sp{-1}x)\cos\theta+\eta(T\sp{-1}x)\sin\theta\right)u_\theta(T\sp{-1}x)\\
& &+\eta(T\sp{-1}x)\cos\theta-\sigma(T\sp{-1}x)\sin\theta\nonumber
\end{eqnarray}

So,  the derivatives of  both function with respect to $\theta$ are

\begin{eqnarray}\label{eq:derf}
\dot{f}_\theta(x)&=&-\lambda(T\sp{-1}x)u_\theta(T\sp{-1}x)\sin\theta \\ \nonumber
& & +\lambda(T\sp{-1}x)\dot u_\theta(T\sp{-1}x)\cos\theta -\lambda(T\sp{-1}x)\cos\theta,
\end{eqnarray}
and
\begin{eqnarray}\label{eq:derg}
\dot{g}_\theta(x)&=&-\sigma(T\sp{-1}x)u_\theta(T\sp{-1}x)\sin\theta+\eta(T\sp{-1}x)u_\theta(T\sp{-1}x)\cos\theta \nonumber \\ \nonumber
& & +(\sigma(T\sp{-1}x)\dot u_\theta((T\sp{-1}x)\cos\theta+\eta(T\sp{-1}x)\dot u_\theta((T\sp{-1}x)\sin\theta\\
& &-\eta(T\sp{-1}x)\sin\theta-\sigma(T\sp{-1}x)\cos\theta.
\end{eqnarray}

Since $u_\theta(x)=f_\theta(x)/g_\theta(x)$,  we can take the derivative with respect to $\theta$ using the expressions \eqref{eq:f},\eqref{eq:g}, \eqref{eq:derf},\eqref{eq:derg} above and, evaluating in $\theta=0$, we obtain

\begin{equation}\label{eqp:09}
\dot u_0(x)=\frac{\lambda(T\sp{-1}x)}{\eta(T\sp{-1}x)}\left(\dot u_0(T\sp{-1}x)-1\right).
\end{equation}

Using the recurrence given by the expression (\ref{eqp:09}) and taking (\ref{eqp:11}) into account, we obtain that for every integer $n\geq 1$

\begin{eqnarray*}
\dot u_0(x)&=&\left(\prod_{j=1}\sp n\frac{\lambda(T\sp{-j}x)}{\eta(T\sp{-j}x)}\right)\cdot\dot u_0(T\sp{-n}x)-\sum_{k=1}\sp n\prod_{j=1}\sp k\frac{\lambda(T\sp{-j}x)}{\eta(T\sp{-j}x)}\\
&=&\frac{\lambda_n(T\sp{-n}x)}{\eta_n(T\sp{-n}x)}\cdot\dot u_0(T\sp{-n}x)-\sum_{k=1}\sp n\frac{\lambda_k(T\sp{-k}x)}{\eta_k(T\sp{-k}x)}.
\end{eqnarray*}

Let $0<\tau<1$ such that $\lambda(x)/\eta(x)<\tau$ for all $x\in X$. Since $u_\theta$ is real analytic, then $\dot u_0(x)$ is bounded and we conclude
\[
\left|\frac{\lambda_n(T\sp{-n}x)}{\eta_n(T\sp{-n}x)}\cdot\dot u_0(T\sp{-n}x)\right| \le\tau\sp k|\dot u_0(T\sp{-n}x)|\to0
\]
when $n\to\infty$. Therefore,
\[
\dot u_0(x)=-\sum_{k=1}\sp \infty\frac{\lambda_k(T\sp{-k}x)}{\eta_k(T\sp{-k}x)}.
\]

Arguing in a similar fashion, we obtain that the second derivative of $u_\theta(x)$ with respect to $\theta$ evaluated in $\theta=0$ is given by
\begin{eqnarray}\label{eqp:10}
\nonumber\ddot u_0(x)&=&\frac{\lambda(T\sp{-1}x)}{\eta(T\sp{-1}x)}\left(\ddot u_0(T\sp{-1}x)-\frac{2\left(\dot u_0(T\sp{-1}x)-1\right)\sp2}{\eta(T\sp{-1}x)}\cdot\sigma(T\sp{-1}x)\right)\\
&=&\frac{\lambda(T\sp{-1}x)}{\eta(T\sp{-1}x)}\left(\ddot u_0(T\sp{-1}x)-\alpha(T\sp{-1}x)\cdot\sigma(T\sp{-1}x)\right).
\end{eqnarray}

Again, the recurrence in (\ref{eqp:10}) allows us to conclude that

\begin{eqnarray*}
\ddot u_0(x)&=&\frac{\lambda_n(T\sp{-n}x)}{\eta_n(T\sp{-n}x)}\cdot\ddot u_0(T\sp{-n}x)-\sum_{k=1}\sp n\frac{\lambda_k(T\sp{-k}x)}{\eta_k(T\sp{-k}x)}\cdot\alpha(T\sp{-k}x)\cdot\sigma(T\sp{-k}x)\\
&=&-\sum_{k=1}\sp \infty\frac{\lambda_k(T\sp{-k}x)}{\eta_k(T\sp{-k}x)}\cdot\alpha(T\sp{-k}x)\cdot\sigma(T\sp{-k}x)
\end{eqnarray*}
\end{proof}

Now, we can study the growth and concavity of $\lambda\sp+(\theta)$ in a neighbourhood of $\theta=0$.

\begin{lema}\label{le:maximo}
If $\sigma(x)=0$ for $\mu$--a.e. $x\in X$ then $\lambda\sp+(\theta)$ has a local maximum in $\theta=0$.
\end{lema}
\begin{proof} From \eqref{eq:pderl}, we have that
\[
\frac{d\lambda\sp+(0)}{d\theta}=0
\]
and from \eqref{eq:sderl} we have
\[
\frac{d\sp2\lambda\sp+(0)}{d\theta\sp2}=\int\left[\frac{2\eta(x)\dot u_0(Tx)}{\lambda(x)}+1\right]d\mu(x).
\]

Note that from \eqref{eqp:12}
\[
-\dot u_0(Tx)=\sum_{k=0}\sp \infty\frac{\lambda_k(T\sp{-k}x)}{\eta_k(T\sp{-k}x)}=\sum_{k=1}\sp \infty\frac{\lambda_k(T\sp{-k}x)}{\eta_k(T\sp{-k}x)}+
\frac{\lambda(x)}{\eta(x)}>\frac{\lambda(x)}{\eta(x)},
\]
and that implies that
\[
\frac{2\eta(x)\dot u_0(Tx)}{\lambda(x)}<-2.
\]
Then,
\[
\frac{2\eta(x)\dot u_0(Tx)}{\lambda(x)}+1<0
\]
and so,
$$\frac{d\sp2\lambda\sp+(0)}{d\theta\sp2}<0.$$
Therefore $\lambda\sp+(\theta)$ has a maximum in $\theta=0$.
\end{proof}

\begin{prop}\label{prp:02}
\[
\frac{d\sp2\lambda\sp+(0)}{d\theta\sp2}<0.
\]
\end{prop}
\begin{proof} Let $(C(X),\|\cdot\|_\infty)$ be the Banach space of all continuous functions $f:X\to\mathbb R$ provided with the supremum norm. Define the linear continuous operator
\[
L(f)=-\sum_{k=1}\sp \infty\left[\frac{\lambda_k}{\eta_k}\cdot\alpha\cdot f\right]\circ T\sp{-k+1}.
\]
Let
$$F_0=\int\left[\frac{2\eta(x)\dot u_0(Tx)}{\lambda(x)}+1\right]d\mu(x).$$
and $F_1:C(X)\to\mathbb R$ defined by
\[
F_1(f)=\int\left[\frac{f(x)L(f)(x)}{\lambda(x)}+\left(\frac{f(x)\dot u_0(Tx)}{\lambda(x)}\right)\sp2\right]d\mu(x).
\]
Then, we consider the  functional $F:C(X)\to\mathbb R$ defined by $F=F_0+F_1$. It is not difficult to see that $F$ is continuous and that
\[
\frac{d\lambda\sp+(0)}{d\theta\sp2}=F(\sigma).
\]
In particular, for $\sigma=0$ we have that $F_1(0)=0$ and from the proof of Lemma~\ref{le:maximo} we obtain that $F(0)<0$. Then there exists $r>0$ such that $F_1(f)<0$ for all $\|f\|<r$. Let $t_0>0$ and $\sigma_0\in C(X)$ such that $\|\sigma_0\|<r$ and $t_0\sigma_0=\sigma$. Then, we have
\[
\frac{d\lambda\sp+(0)}{d\theta\sp2}=F(\sigma)=F_0+F_1(t_0\sigma_0)=F_0+t_0\sp2F_1(\sigma_0)<0
\]
as desired.
\end{proof}
\section{Heisenberg Nilmanifold}\label{sec:appl}

Let $\mathcal{H}\equiv\R\sp 3$ be the Heisenberg group of upper
triangular $3\times 3$ matrices with ones in the diagonal and consider
$\Phi\::\:\Hei\to\Hei$ be the automorphism  defined by
$\Phi(\x,y)=(B\x,l(\x,y))$, where $\x=(x_1,x_2)$ and
\begin{equation}\label{eq:automorf}
B=\left(\begin{array}{rccl} 2 & 1 \\ 1 & 1 \end{array} \right) \mbox{
and }l({\bf x},y)= y+x_1^2+\frac12x_2^2+x_1x_2.
\end{equation}

Let $N$ be the compact nilmanifold obtained from the quotient of the
Heisenberg group $\Hei$ by the lattice $\Gamma=\{(\x, y) : \x\in
\mathbb{Z}^{2}, y\in \frac1{2}\mathbb{Z}\}$. Let $\pi\::\:\Hei\to N$ be the
projection and $f\::\:N\to N$ be the induced diffeomorphism from $\Phi$
by
$$
f\circ \pi=\pi\circ\Phi.
$$
Then $f$ is a strong partially hyperbolic diffeomorphism. We remark that $f$ is ergodic with respect to Lebesgue
measure with zero central Lyapunov exponent.

We recall that the group $\mathcal{H}$ is a Lie group with the usual
product of matrices, then the automorphism $\Phi$ is conjugated via the
exponential map (of the Lie algebra) with $D\Phi(({\bf 0},0))$. Therefore, the
stable, unstable and central leaves of $\Phi$ at the point
$(\x,y)\in\mathcal{H}$ are explicitly
$$
W\sp s((\x, y),\Phi)=\left\{\left(t\v\sp s_B,\frac{tx_1p_2+tx_2p_1+t\sp
2p_1p_2}{2}\right)\::\:t\in\R\right\},
$$
$$
W\sp u((\x, y),\Phi)=\left\{\left(t\v\sp u_B,\frac{tx_1q_2+tx_2q_1+t\sp
2q_1q_2}{2}\right)\::\:t\in\R\right\}
$$
and
$$
W\sp c((\x, y),\Phi)=\left\{\left(\x,y+t\right)\::\:t\in\R\right\},
$$
where $\v\sp s_B=(p_1,p_2)$ and $\v\sp u_B=(q_1,q_2)$ are the
eigenvector associated to the eigenvalues $\lambda\sp s_B$ and
$\lambda\sp u_B$ respectively. Thus, the invariant spaces for $\Phi$
at $(\x,y)$ are given by
$$
E\sp s_\Phi((\x, y))=\left\langle\left(\v\sp
s,\frac{x_1p_2+x_2p_1}{2}\right)\right\rangle,$$
$$
E\sp u_\Phi((\x, y))=\left\langle\left(\v\sp
u,\frac{x_1q_2+x_2q_1}{2}\right)\right\rangle
$$
and
$$
E\sp c_\Phi((\x, y))=\langle({\bf 0},1)\rangle.
$$

The sets described above are in fact leaves of invariant foliations and
they are projected onto the invariant foliations in $N$. Observe that
$W^{c}((\x,y),\Phi)$ is projected onto a circle in $N$ and hence the
projection of the central foliation is a foliation by circles.
Moreover, these circles are collapsed by the projection $p_{N}$ and
hence the central foliation is a nontrivial fibration with base
$\mathbb{T}^{2}$ and fiber $\mathbb{S}^{1}$.

Recall that $E^{u}_\Phi((\x,y))$ is the space generated by the vector
$(\v\sp u,\frac {x_{1}q_{2}+x_{2}q_{1}}2)$ and  $E^{c}_\Phi((\x,y))$ is
the space generated by the vector $({\bf 0},1)$. Then,
$E\sp{cu}_\Phi((\x,y))$ is generated by the vectors $(\v\sp u,0)$ and
$({\bf 0},1)$, so $E\sp{cu}_\Phi$ does not depend of the point $(\x,y)$.
The derivative of $\Phi$ is given by
\[
D\Phi(\x,y)=\left(
  \begin{array}{ccc}
    2 & 1 & 0 \\
    1 & 1 & 0 \\
    2x_1+x_2 & x_1+x_2 & 1 \\
  \end{array}
\right).
\]
Evaluating $D\Phi(\x,y)$ in the vectors $(\v\sp u,0)$ and $({\bf 0},1)$ we obtain
\[
D\Phi(\x,y)(\v\sp u,0)=(\lambda_A\sp u\v\sp u, A\x\cdot \v\sp
u)=(\lambda_A\sp u\v\sp u, \x\cdot A\v\sp u)=\lambda_A\sp u(\v\sp
u,0)+\lambda_A\sp u\x\cdot \v\sp u({\bf 0},1)
\]
and
\[
D\Phi(\x,y)({\bf 0},1)=({\bf 0},1).
\]
Then, we obtain $D\Phi(\x,y)|{E\sp{cu}_\Phi}$ in terms of the base
$\{(\v\sp u,0), ({\bf 0},1)\}$,
\begin{equation}\label{eq:derivadaphi}
D\Phi(\x,y)|{E\sp{cu}_\Phi}=\left(\begin{array}{cc}
                         \lambda_A\sp u & 0 \\
                         \lambda_A\sp ux\cdot \v\sp u & 1
                       \end{array}\right).
\end{equation}
Clearly $E\sp{cu}_f(\pi(\x,y))=D\pi(\x,y)E\sp{cu}_\Phi$. We
assert that the sub--bundle $E\sp{cu}_f\subset TN$ is trivial. In fact,
observe that a fundamental domain for $N=\mathcal{H}/\Gamma$ is the cube
\[
\mathcal{C}=\{(\x,y)\in\mathcal{H}:0\le x_1, x_2\le 1,\, 0\le y\le
1/2\}.
\]
Since the product in $\mathcal{H}$ can be written by $(\x,y)\cdot
(\a,b)=(\x+\a, y+ b+x_1a_2)$, then in the box $\mathcal{C}$, we identify
the top face with the bottom face because
\[
(\x, 0)\cdot({\bf 0}, 1/2)= (\x, 1/2)\, \Rightarrow\, (\x, 0)\sim (\x,
1/2),
\]
and the front face with the back face because
\[
((0,x_2),y)\cdot({\bf e}_1, 0)= ((1,x_2), y)\, \Rightarrow\,
((0,x_2),y)\sim ((1,x_2), y).
\]
The group multiplication
\[
(\x, y)\cdot ({\bf e}_2, 0)= ((x_1,x_2+1), y+x_1)
\]
tells us to join left and right faces of the cube by the identification
\begin{equation}\label{eq:03}
((x_1, 0), y)\sim\left((x_1, 1), y+x_1 \mod \frac{1}{2}\right).
\end{equation}
Thus $N=\mathcal{C}/\sim$, that is, the nilmanifold is the cube
$\mathcal{C}$ with the previous relations. To see a representation of
$E_f\sp{cu}$ it is necessary to give relations in the tangents of
$E_\Phi\sp{cu}$ restricted to the border of the cube.

\noindent For top (resp. bottom) faces and front (resp. back) faces, the relation in the vectors
is the identity. On the other hand, let $h$ be the right multiplication
of $(\x,y)$ by $({\bf e}_2, 0)$. Since
\[
Dh(\x,y)=\left(
           \begin{array}{ccc}
             1 & 0 & 0 \\
             0 & 1 & 0 \\
             1 & 0 & 1 \\
           \end{array}
         \right)
\]
then $Dh((x_1, 0), y)({\bf 0},1)=({\bf 0},1)$ and $Dh((x_1, 0), y)({\bf
v}\sp u_B,0)=({\bf v}\sp u_B,q_1)$. It follows from the
relation~\ref{eq:03} and the last equalities the following vector
relations:
\[
({\bf 0},1)\in T_{((x_1, 0), y)}\mathcal{H}\, \sim ({\bf 0},1)\in
T_{((x_1, 1), y+x_1 \mod 1/2)}\mathcal{H}
\]
and
\[
({\bf v}\sp u_B,0)\in T_{((x_1, 0), y)}\mathcal{H}\, \sim ({\bf v}\sp
u_B,q_1)\in T_{((x_1, 1), y+x_1 \mod 1/2)}\mathcal{H}.
\]
\noindent Thus $ E\sp{cu}_f=\left(\bigsqcup_{w\in
\mathcal{C}}E_{\Phi,w}\sp{cu}\right)/\sim$. To see that $E_f\sp{cu}$ is
trivial, observe that the sections $v_1$ and $v_2$ over $E_f\sp{cu}$
defined by
\[
v_1((x_1,x_2),y)=({\bf 0},1)
\]
and
\[
v_2((x_1,x_2),y)=({\bf v}_B\sp{u},x_2q_1)
\]
are global non--zero sections that define a base in each fiber of
$E_f\sp{cu}$.

Finally, we define the cocycle $F:TN\to TN$ as the cocycle induced by $f$ and its derivative, that is,
$$F((\x,y),v)=(f(\x,y),Df(\x,y)v).$$

\noindent and we define the one--parameter family of continuous cocycles $F_\theta\::\: TN\to TN$ defined by
$$
F_\theta((\x,y),v)=\begin{cases} Df(\x,y)R_\theta v&, \quad v\in E^{cu};\\
Df(\x,y)v &, \quad v\in
E^s.
\end{cases}
$$

\noindent From item (i) of Theorem~\ref{teo:main}, there exists an open set $\mathcal{I}\subset\mathbb{R}$ such that  for every $\theta\in\mathcal{I}$, the cocycle $(F_\theta|E\sp{cu})$ has dominated splitting. Moreover $(F_\theta|E\sp{cu})$ is partially hyperbolic, and so is the cocycle $F_\theta$ whose splitting is given by $TN=E\sp s\oplus E^c_\theta \oplus E^u_\theta$. From item (ii) of Theorem~\ref{teo:main}, reducing the open set $\mathcal{I}$ if necessary,  the central Lyapunov exponent of $F_\theta$ is positive and then we prove the Corollary~\ref{cor:main}.


\bibliographystyle{plain}

\bibliography{VV}

\def\cprime{$'$} \def\cprime{$'$}
\begin{thebibliography}{10}

\bibitem{AB}
Artur Avila and Jairo Bochi.
\newblock A formula with some applications to the theory of {L}yapunov
  exponents.
\newblock {\em Israel J. Math.}, 131:125--137, 2002.

\bibitem{BLD2011}
Alexandre~T. Baraviera, Jo{\~a}o Lopes~Dias, and Pedro Duarte.
\newblock On the {H}erman-{A}vila-{B}ochi formula for {L}yapunov exponents of
  {${\rm SL}(2,\Bbb R)$}-cocycles.
\newblock {\em Nonlinearity}, 24(9):2465--2476, 2011.

\bibitem{BC2004}
Christian Bonatti and Sylvain Crovisier.
\newblock R\'ecurrence et g\'en\'ericit\'e.
\newblock {\em Invent. Math.}, 158(1):33--104, 2004.

\bibitem{BDV05}
Christian Bonatti, Lorenzo~J. D{\'{\i}}az, and Marcelo Viana.
\newblock {\em Dynamics beyond uniform hyperbolicity}, volume 102 of {\em
  Encyclopaedia of Mathematical Sciences}.
\newblock Springer-Verlag, Berlin, 2005.
\newblock A global geometric and probabilistic perspective, Mathematical
  Physics, III.

\bibitem{FK}
H.~Furstenberg and H.~Kesten.
\newblock Products of random matrices.
\newblock {\em Ann. Math. Statist.}, 31:457--469, 1960.

\bibitem{F1963}
Harry Furstenberg.
\newblock Noncommuting random products.
\newblock {\em Trans. Amer. Math. Soc.}, 108:377--428, 1963.

\bibitem{H}
Michael-R. Herman.
\newblock Une m\'ethode pour minorer les exposants de {L}yapounov et quelques
  exemples montrant le caract\`ere local d'un th\'eor\`eme d'{A}rnol\cprime d
  et de {M}oser sur le tore de dimension {$2$}.
\newblock {\em Comment. Math. Helv.}, 58(3):453--502, 1983.

\bibitem{Kn92}
Oliver Knill.
\newblock Positive {L}yapunov exponents for a dense set of bounded measurable
  {${\rm SL}(2,{\bf R})$}-cocycles.
\newblock {\em Ergodic Theory Dynam. Systems}, 12(2):319--331, 1992.

\bibitem{M}
Ricardo Ma{\~n}{\'e}.
\newblock Oseledec's theorem from the generic viewpoint.
\newblock In {\em Proceedings of the {I}nternational {C}ongress of
  {M}athematicians, {V}ol.\ 1, 2 ({W}arsaw, 1983)}, pages 1269--1276, Warsaw,
  1984. PWN.

\bibitem{O}
V.~I. Oseledec.
\newblock A multiplicative ergodic theorem. {C}haracteristic {L}japunov,
  exponents of dynamical systems.
\newblock {\em Trudy Moskov. Mat. Ob\v s \v c.}, 19:179--210, 1968.

\bibitem{Ruelle79}
D.~Ruelle.
\newblock Analycity properties of the characteristic exponents of random matrix
  products.
\newblock {\em Adv. in Math.}, 32(1):68--80, 1979.

\bibitem{SW}
Michael Shub and Amie Wilkinson.
\newblock Pathological foliations and removable zero exponents.
\newblock {\em Invent. Math.}, 139(3):495--508, 2000.

\bibitem{pancho2}
F.~{Valenzuela}.
\newblock {On Critical Point for Two Dimensional Holomorphics Systems}.
\newblock {\em ArXiv e-prints}, May 2011.

\bibitem{Yoccoz2004}
Jean-Christophe Yoccoz.
\newblock Some questions and remarks about {${\rm SL}(2,\bold R)$} cocycles.
\newblock In {\em Modern dynamical systems and applications}, pages 447--458.
  Cambridge Univ. Press, Cambridge, 2004.

\end{thebibliography}

\end{document}